\documentclass[11pt]{amsart}
\usepackage{amsmath,amssymb,amsfonts,amsthm,amsrefs}
\usepackage[margin=1in]{geometry}
\ifpdf
\usepackage{hyperref}%[unicode]{hyperref}
\else
\usepackage[hypertex]{hyperref}
\fi
\usepackage{graphicx}
\title{Comparing the degrees of enumerability and the closed Medvedev degrees}
\author[P.~Shafer]{Paul Shafer}
\address{School of Mathematics\\
University of Leeds\\
Leeds, LS2 9JT, United Kingdom}
\email{\href{mailto:p.e.shafer@leeds.ac.uk}{p.e.shafer@leeds.ac.uk}}
\urladdr{\url{http://www1.maths.leeds.ac.uk/~matpsh/}}
\thanks{The first author was supported by EPSRC Overseas Travel Grant
No.~EP/R006458/1}

\author[A.~Sorbi]{Andrea Sorbi}
\address{Dipartimento di Ingegneria Informatica e Scienze Matematiche\\
Universit\`a Degli Studi di Siena\\
I-53100 Siena, Italy}
\email{\href{mailto:andrea.sorbi@unisi.it}{andrea.sorbi@unisi.it}}
\urladdr{\url{http://www3.diism.unisi.it/~sorbi/}}
\thanks{The second author is a member of INDAM-GNSAGA}



\newtheorem{Theorem}{Theorem}%[section]
\newtheorem{Lemma}[Theorem]{Lemma}

\newtheorem{Proposition}[Theorem]{Proposition}

\theoremstyle{definition}
\newtheorem{Definition}[Theorem]{Definition}
\newtheorem{Remark}[Theorem]{Remark}

\newcommand{\la}{\langle}
\newcommand{\ra}{\rangle}
\newcommand{\join}{\vee}
\newcommand{\meet}{\wedge}
\newcommand{\orr}{\text{ or }}
\newcommand{\andd}{\text{ and }}
\newcommand{\imp}{\rightarrow}

\newcommand{\biimp}{\leftrightarrow}

\newcommand{\da}{{\downarrow}}

\newcommand{\smf}{\smallfrown}

\DeclareMathOperator{\graph}{graph}
\DeclareMathOperator{\dom}{dom}
\DeclareMathOperator{\leaves}{leaves}

\DeclareMathOperator{\ran}{ran}
\DeclareMathOperator{\dg}{deg}

\newcommand{\dnr}{\mathrm{DNR}}

\newcommand{\leqT}{\leq_{\mathrm T}}

\newcommand{\midT}{\mid_{\mathrm T}}

\newcommand{\degT}{\dg_{\mathrm T}}

\newcommand{\leqs}{\leq_{\mathrm s}}

\newcommand{\les}{<_{\mathrm s}}

\newcommand{\nleqs}{\nleq_{\mathrm s}}

\newcommand{\equivs}{\equiv_{\mathrm s}}
\newcommand{\nequivs}{\not\equiv_{\mathrm s}}
\newcommand{\degs}{\dg_{\mathrm s}}

\newcommand{\leqw}{\leq_{\mathrm w}}

\newcommand{\nleqw}{\nleq_{\mathrm w}}

\newcommand{\skp}{\diamond}

\newcommand{\leqe}{\leq_\mathrm{e}}
\newcommand{\nleqe}{\nleq_\mathrm{e}}

\newcommand{\equive}{\equiv_\mathrm{e}}

\newcommand{\gee}{>_\mathrm{e}}

\newcommand{\dege}{\dg_{\mathrm{e}}}

\newcommand{\MP}[1]{\ensuremath{\mathcal{#1}}}
\newcommand{\md}[1]{\ensuremath{\mathbf{#1}}}

\usepackage{xcolor}	
\usepackage{soul}

\AtBeginDocument{%
   \def\MR#1{}
}

\begin{document}

\begin{abstract}
We compare the degrees of enumerability and the closed Medvedev degrees and
find that many situations occur.  There are nonzero closed degrees that do
not bound nonzero degrees of enumerability, there are nonzero degrees of
enumerability that do not bound nonzero closed degrees, and there are
degrees that are nontrivially both degrees of enumerability and closed
degrees.  We also show that the compact degrees of enumerability exactly
correspond to the cototal enumeration degrees.
\end{abstract}

\maketitle

\section*{Introduction}

The purpose of this work is to explore the distribution of the so-called
\emph{degrees of enumerability} with respect to the closed degrees within the
Medvedev degrees.  Both the enumeration degrees and the Turing degrees embed
into the Medvedev degrees.  The Medvedev degrees corresponding to enumeration
degrees are called \emph{degrees of enumerability}, and the Medvedev degrees
corresponding to Turing degrees are called \emph{degrees of solvability}.
The embedding of the Turing degrees into the Medvedev degrees is particularly
nice.  The degrees of solvability are all closed (being the degrees of
singleton sets), and the collection of all degrees of solvability is
definable in the Medvedev degrees.  On the other hand, whether the degrees of
enumerability are definable in the Medvedev degrees is a longstanding open
question of Rogers~\cites{Rogers:Agenda, Rogers:Book}.

In light of Roger's question and the nice definability and topological
properties of the degrees of solvability, we find it natural to investigate
the behavior of the degrees of enumerability with respect to the closed
degrees.  Together, our main results show that the relation between the
degrees of enumerability and the closed degrees is considerably more nuanced
than the relation between the degrees of solvability and the closed degrees.

\begin{itemize}
\item There are nonzero closed degrees that do not bound nonzero degrees of
    enumerability.  In fact, there are nonzero degrees that are closed,
    uncountable, and meet-irreducible that do not bound nonzero degrees of
    enumerability (Proposition~\ref{prop-ClosedNotEnum}).

\smallskip

\item  There are nonzero closed (indeed, compact) degrees of enumerability
    that do not bound nonzero degrees of solvability
    (Theorem~\ref{thm:qm-compact}). Moreover, the compact degrees of
    enumerability exactly correspond to the cototal enumeration degrees
    (Theorem~\ref{thm:cototal=compact}).

\smallskip

\item There are nonzero degrees of enumerability that do not bound nonzero
    closed degrees (Theorem~\ref{thm:no-bounding}).
\end{itemize}

We work in Baire space and interpret an arbitrary set $\MP A \subseteq
\omega^\omega$ as representing an abstract mathematical problem, namely the
problem of finding (or, computing) a member of $\MP A$. For this reason, we
refer to subsets of Baire space as \emph{mass problems}. For sets $\MP{A},
\MP{B} \subseteq \omega^\omega$, we say that \emph{$\MP A$ Medvedev (or
strongly) reduces to $\MP B$}, and we write $\MP A \leqs \MP B$, if there is
a Turing functional $\Phi$ such that $\Phi(\MP B) \subseteq \MP A$, meaning
that $\Phi(f)$ is total and is in $\MP A$ for every $f \in \MP B$. Under the
interpretation of subsets of Baire space as mathematical problems, $\MP A
\leqs \MP B$ means that problem $\MP B$ is at least as hard as problem $\MP
A$ in a computational sense because every solution to problem $\MP B$ can be
converted into a solution to problem $\MP A$ by a uniform computational
procedure.

Medvedev reducibility induces an equivalence relation called \emph{Medvedev
(or strong) equivalence} in the usual way:  $\MP A \equivs \MP B$ if and only
if $\MP A \leqs \MP B$ and $\MP B \leqs \MP A$.  The $\equivs$-equivalence
class $\degs(\MP A) = \{\MP B : \MP B \equivs \MP A\}$ of a mass problem $\MP
A$ is called its \emph{Medvedev (or strong) degree}, and the collection of
all such equivalence classes, ordered by Medvedev reducibility, is a
structure called the \emph{Medvedev degrees}.  The Medvedev degrees form a
bounded distributive lattice (in fact, a Brouwer algebra), with least element
$\md 0 = \{\MP A : \text{$\MP A$ has a recursive member}\}$ and greatest
element $\md 1 = \{\emptyset\}$.  Joins and meets in the Medvedev degrees are
computed as follows:
\begin{align*}
\degs(\MP A) \join \degs(\MP B) &= \degs(\MP A \oplus \MP B)\\
\degs(\MP A) \meet \degs(\MP B) &= \degs(0^\smf\MP A \cup 1^\smf\MP B).
\end{align*}
For joins, $\MP A \oplus \MP B = \{f \oplus g : f \in \MP A \andd g \in \MP
B\}$, where $f \oplus g$ is the usual Turing join of $f$ and $g$:  $(f \oplus
g)(2n) = f(n)$ and $(f \oplus g)(2n+1) = g(n)$.  For meets, $0^\smf\MP A \cup
1^\smf\MP B$ is the set obtained by prepending $0$ to every function in $\MP
A$, prepending $1$ to every function in $\MP B$, and taking the union of the
resulting sets.

Under the  interpretation of mass problems as mathematical problems, problem
$\MP A \oplus \MP B$ corresponds to the problem of solving problem $\MP A$
\emph{and} solving problem $\MP B$, and problem $0^\smf\MP A \cup 1^\smf\MP
B$ corresponds to the problem of solving problem $\MP A$ \emph{or} solving
problem $\MP B$.  Medvedev introduced the structure that now bears his name
in \cite{Medvedev1955}.  For an introduction to the Medvedev degrees,
including its origins and motivation, see \cite{Rogers:Book}*{Chapter~13.7}.
For surveys on the Medvedev degrees and related topics, see
\cites{SorbiSurvey, HinmanSurvey}.  For recursive aspects of the Medvedev
degrees, see \cite{Dyment1976Rus}.  For algebraic aspects of the Medvedev
degrees and applications to intermediate logics, see for instance
\cites{Skvortsova:intuitionismEng, Sorbi:Brouwer, Kuyper-Medvedev}.

\subsection*{Notation}
We use the following notation and terminology regarding strings and trees.
Denote by $\omega^{<\omega}$ the set of all finite strings of natural
numbers, and denote by $2^{<\omega}$ the set of all finite binary strings.
For $\sigma \in \omega^{<\omega}$, $|\sigma|$ denotes the length of $\sigma$.
We denote the empty string by $\emptyset$.  For $\sigma, \tau \in
\omega^{<\omega}$, $\sigma \subseteq \tau$ means that $\sigma$ is an initial
segment of $ \tau$, and $\sigma^\smf\tau$ denotes the concatenation of
$\sigma$ and $\tau$. Similarly, for $ \sigma \in \omega^{<\omega}$ and $f \in
\omega^\omega$, $\sigma \subset f$ means that $ \sigma$ is an initial segment
of $f$, i.e., $(\forall n < |\sigma|)(f(n) = \sigma(n))$, and $\sigma^\smf f$
denotes the concatenation of $\sigma$ and $f$:
\begin{align*}
(\sigma^\smf f)(n) =
\begin{cases}
\sigma(n) & \text{if $n < |\sigma|$}\\
f(n-|\sigma|) & \text{if $n \geq |\sigma|$}.
\end{cases}
\end{align*}
If $\sigma \in \omega^{<\omega}$ and $f \in \omega^\omega$,
$\sigma \# f$ denotes the result of replacing the initial segment of $f$ of
length $|\sigma|$ by $\sigma$:
\begin{align*}
(\sigma \# f)(n) =
\begin{cases}
\sigma(n) & \text{if $n < |\sigma|$}\\
f(n) & \text{if $n \geq |\sigma|$}.
\end{cases}
\end{align*}
For $\sigma \in \omega^{<\omega}$ and $\MP A \subseteq \omega^\omega$, we
define $\sigma^\smf\MP A = \{\sigma^\smf f : f \in \MP A\}$ and $\sigma\#\MP
A = \{\sigma \# f : f \in \MP A\}$. Finally, for $\sigma \in
\omega^{<\omega}$ and $n \leq |\sigma|$, $\sigma\!\restriction\!n$ denotes
the initial segment $\la \sigma(0), \dots, \sigma(n-1) \ra$ of $\sigma$ of
length $n$. Similarly, for $f \in \omega^\omega$ and $n \in \omega$,
$f\!\restriction\!n$ denotes the initial segment $\la f(0), \dots, f(n-1)
\ra$ of $f$ of length $n$.

A \emph{tree} is a set $T \subseteq \omega^{<\omega}$ that is closed under
initial segments: $(\forall \sigma, \tau \in \omega^{<\omega})((\sigma
\subseteq \tau \andd \tau \in T) \imp \sigma \in T)$.  A node $\sigma$ in a
tree $T$ is a \emph{leaf} if there is no $\tau \supset \sigma$ with $\tau \in
T$.  A tree $T$ is \emph{finitely branching} if for every $\sigma \in T$
there are at most finitely many strings $\tau \in T$ with $|\tau| = |\sigma|
+ 1$.   A string $\sigma \in \omega^{<\omega}$ is \emph{bounded} by an $h \in
\omega^\omega$ (or \emph{$h$-bounded}) if $(\forall n < |\sigma|)(\sigma(n) <
h(n))$.  Likewise, a tree $T$ is \emph{$h$-bounded} if $(\forall \sigma \in
T)(\text{$\sigma$ is $h$-bounded})$.  For $b \in \omega$, $b$-bounded means
bounded by the function that is constantly $b$. An $f \in \omega^\omega$ is
an infinite path through a tree $T$ if $\forall n (f\!\restriction\!n \in
T)$. The subset of Baire space consisting of all infinite paths through a
tree $T$ is denoted by $[T]$.  The closed subsets of Baire space are exactly
those of the form $[T]$ for a tree $T$, and the compact subsets of Baire
space are exactly those of the form $[T]$ for a finitely branching tree $T$.

Throughout, we refer to a standard listing $(\Phi_e: e \in \omega)$ of all
Turing functionals on Baire space. If $\Phi$ is a Turing functional and
$\sigma$ is a finite string of natural numbers, then $\Phi(\sigma)$ denotes
the longest string $\tau$ such that $(\forall m < |\tau|)(\tau(m) =
\Phi(\sigma)(m)\da)$). We also refer to a standard listing $(\Psi_e: e \in
\omega)$ of all enumeration operators. If $\Psi$ is an enumeration operator
and $\phi$ is a partial function, then $\Psi(\phi)$ stands for
$\Psi(\graph(\phi))$.  Recall that $\la \cdot, \cdot \ra \colon \omega^2 \imp
\omega$ is the usual recursive Cantor pairing function and that $\graph(\phi)
= \{\la n, y \ra : n \in \dom(\phi) \andd \phi(n) = y\}$.

For further background concerning recursion theory, trees, and the topology
of Baire space, we refer the reader to standard textbooks such as
\cites{Rogers:Book, Moschovakis:Descriptive}.

\subsection*{Degrees of solvability and degrees of enumerability}
As discussed in the introduction, part of the interest in the Medvedev
degrees comes from the fact that the structure embeds both the Turing degrees
and the enumeration degrees.  Singleton subsets of Baire space are called
\emph{problems of solvability}, and their corresponding Medvedev degrees are
called \emph{degrees of solvability}.  It is easy to see that the assignment
$\degT(f) \mapsto \degs(\{f\})$ embeds the Turing degrees into the Medvedev
degrees, preserving joins and the least element, and that the range of this
embedding is exactly the degrees of solvability.  Moreover, the degrees of
solvability are definable in the Medvedev degrees \cites{Medvedev1955,
Dyment1976Rus} (see also \cites{SorbiSurvey, Rogers:Book}).

To embed the enumeration degrees into the Medvedev degrees, given a nonempty
$A \subseteq \omega$, let
\begin{align*}
\MP{E}_A = \{f : \ran(f) = A\}.
\end{align*}
$\MP{E}_A$ is called the \emph{problem of enumerability of $A$}, and it
represents the problem of enumerating the set $A$.  The corresponding
Medvedev degree $\md{E}_A = \degs(\MP{E}_A)$ is called the \emph{degree of
enumerability of $A$}.  For nonempty $A, B \subseteq \omega$, it is easy to
see that $A \leqe B$ if and only if $\MP{E}_A \leqs \MP{E}_B$.  This gives
rise to an embedding $\dege(A) \mapsto \md{E}_A$ of the enumeration degrees
into the Medvedev degrees.  The embedding preserves joins and the least
element, and the range of the embedding is exactly the degrees of
enumerability \cite{Medvedev1955} (see also \cites{SorbiSurvey,
Rogers:Book}). Again we mention that, contrary to definability of the degrees
of solvability, it is still an open question (see
Rogers~\cites{Rogers:Agenda, Rogers:Book}) whether the degrees of
enumerability are definable, or at least invariant under automorphisms, in
the Medvedev degrees.

The following lemma (which we state and prove for later reference) is
well-known.  It corresponds to the fact that the Turing degrees embed (again
via an embedding that preserves joins and the least element) into the
enumeration degrees of total functions.

\begin{Lemma}\label{lem-EnumTotGraph}
If $f \colon \omega \imp \omega$ is total, then $\MP{E}_{\graph(f)} \equivs
\{f\}$.
\end{Lemma}

\begin{proof}
Clearly $\MP{E}_{\graph(f)} \leqs \{f\}$ via the Turing functional
$\Phi(f)(n)=\langle n, f(n)\rangle$. To see that $\{f\} \leqs
\MP{E}_{\graph(f)}$, let $\Gamma$ be the Turing functional such that, for
every total $g \colon \omega \imp \omega$ and $n \in \omega$, $\Gamma(g)(n)$
searches for the least $k$ such that $g(k) = \la n, y \ra$ for some $y$, and
outputs $y$. Then $\Gamma(g) = f$ whenever $\ran(g) = \graph(f)$, so $\{f\}
\leqs \MP{E}_{\graph(f)}$.
\end{proof}

In analogy with the common terminology used in the enumeration degrees, we
say that a problem of enumerability $\MP E$ is \emph{total} if $\MP E \equivs
\{f\}$ for some total $f$.  That is, a problem of enumerability is total if
it is Medvedev-equivalent to a problem of solvability.  Likewise, we say that
a degree of enumerability is \emph{total} if it is the Medvedev degree of a
total problem of enumerability.  Now recall that an $A \subseteq \omega$ is
\emph{quasiminimal} if $A$ is not r.e.\ and there is no nonrecursive total
$f$ with $f \leqe A$ (meaning, as usual, that there is no nonrecursive total
$f$ with $\graph(f) \leqe A$). We say that a problem of enumerability $\MP E$
is \emph{quasiminimal} if $\MP E \equivs \MP{E}_A$ for a quasiminimal $A$.
Likewise, we say that a degree of enumerability is \emph{quasiminimal} if it
is the Medvedev degree of a quasiminimal problem of enumerability.
Lemma~\ref{lem-EnumTotGraph} implies the following lemma.

\begin{Lemma}\label{lem:withoutproof}
If $\MP{E}_A$ is a quasiminimal problem of enumerability, then $\md 0 \les
\md{E}_A$ and $\MP{E}_A \nequivs \{f\}$ for every total $f$ (in fact $\{f\}
\nleqs \MP{E}_A$ for every nonrecursive total $f$).
\end{Lemma}

Both the degrees of solvability and the degrees of enumerability enjoy the
algebraic property of \emph{meet-irreducibility}.  Recall that an element $a$
of a lattice $L$ is called \emph{meet-reducible} if it is the meet of a pair
of strictly larger elements:  $(\exists b,c  \in L)(b > a \andd c > a \andd a
= b \meet c)$.  An element of a lattice is called \emph{meet-irreducible} if
it is not meet-reducible.  It is well-known that, in a distributive lattice
$L$ such as the Medvedev degrees, an element $a$ is meet-irreducible if and
only if $(\forall b,c \in L)((a \geq b \meet c) \imp (a \geq b \orr a \geq
c))$ (see \cite{BalbesDwinger}*{Section III.2}).

We now recall some helpful terminology and a lemma before proving that the
degrees of solvability and the degrees of enumerability are meet-irreducible.
These facts are known in the literature, but we include proofs for the sake
of completeness.  For a mass problem $\MP A$ and a $\sigma \in
\omega^{<\omega}$, let $\MP{A}_\sigma=\{f \in \MP{A} : \sigma \subset f\}$.
Call a mass problem $\MP A$ \emph{uniform} if $\MP{A}_\sigma \leqs \MP{A}$
whenever $\sigma \in \omega^{<\omega}$ is such that $\sigma \subset f$ for
some $f \in \MP A$.

\begin{Lemma}\cite{Dyment1976Rus}*{Corollary 2.8}\label{lem:uniform-meet-irr}
Every uniform mass problem has meet-irreducible Medvedev degree.
\end{Lemma}
\begin{proof}
Suppose that $\MP A$ is a uniform mass problem and that $\MP B$ and $\MP C$
are arbitrary mass problems such that $0^\smf\MP B \cup 1^\smf\MP C \leqs \MP
A$. We may assume that $\MP A \neq \emptyset$ as clearly $\md 1$ is
meet-irreducible. Let $\Phi$ be such that $\Phi(\MP A) \subseteq 0^\smf\MP B
\cup 1^\smf\MP C$. Choose any $f \in \MP A$, and let $ \sigma \subset f$ be
such that $\Phi(\sigma)(0)\da$. Let $b = \Phi(\sigma)(0)$, and observe that
$b \in \{0,1\}$.  Suppose for the sake of argument that $b = 0$. Then, as
every $f \in \MP{A}_\sigma$ begins with $\sigma$ and is in $\MP{A}$, we have
that $\Phi(f)(0)=0$ for every $f \in \MP{A}_\sigma$, thus yielding $\MP B
\leqs 0^\smf\MP B \leqs \MP A_\sigma \leqs \MP A$. Similarly, if $b = 1$,
then $\MP C \leqs \MP A$.  Thus either $\MP B \leqs \MP A$ or $\MP C \leqs
\MP A$. So $\MP A$ has meet-irreducible degree.
\end{proof}

\begin{Proposition}[\cites{Medvedev1955, Sorbi:Someremarks}]\label{prop-MeetIrred}
In the Medvedev degrees, every degree of solvability is meet-irreducible, and
every degree of enumerability is meet-irreducible.
\end{Proposition}

\begin{proof}
It suffices to prove that the degrees of enumerability are meet-irreducible
(see \cite{Sorbi:Someremarks}*{Theorem~4.5}) because every degree of
solvability is also a degree of enumerability.  (It is also easy to simply
observe that if $0^\smf\MP B \cup 1^\smf\MP C \leqs \{f\}$, then either $\MP
B \leqs \{f\}$ or $\MP C \leqs \{f\}$.)

Let $\md{E}_A$ be the degree of enumerability of $A \subseteq \omega$.  The
proposition follows from Lemma~\ref{lem:uniform-meet-irr}, as it is easy to
see that $\MP{A}=\MP{E}_A$ is uniform: if $\sigma \subset f$ for some $f \in
\MP{E}_A$, just consider the reduction $\MP{A}_\sigma \leqs \MP{A}$ given by
$\Phi(f)=\sigma^\smf f$.
\end{proof}

In a similar spirit, Dyment proved that if $\MP B$ is a countable (or finite)
mass problem, if $ \MP{E}_A$ is the problem of enumerability of $A \subseteq
\omega$, and if $\MP B \leqs \MP{E} _A$, then there is a $g \in \MP B$ such
that $g \leqe A$ \cite{Dyment1976Rus}*{Theorem 3.4}.  Call a Medvedev degree
\emph{countable} if it is the degree of a countable (or finite) mass problem,
and call it \emph{uncountable} otherwise.  Dyment's result implies that if
$\md E$ is a nontotal degree of enumerability, then $\md E$ is uncountable
\cite{Dyment1976Rus}*{Corollary 3.14}.

\section*{Comparing degrees of enumerability and closed degrees}

A Medvedev degree is called \emph{closed} if it is of the form $\degs(\MP C)$
for a closed $\MP C \subseteq \omega^\omega$.  Every degree of solvability is
closed because singletons are closed.  Thus there are closed degrees of
enumerability because every degree of solvability is also a degree of
enumerability.  It is, however, easy to produce examples of closed degrees
that are not degrees of enumerability.  Let $f,g \in \omega^\omega$ be such
that $f \midT g$.  Then $ \degs(\{f,g\})$ is closed, but it is not a degree
of enumerability because it is meet-reducible (as $ \degs(\{f,g\}) =
\degs(\{f\}) \meet \degs(\{g\})$), whereas all degrees of enumerability are
meet- irreducible by Proposition~\ref{prop-MeetIrred}.  In fact, by the
discussion following Proposition~\ref{prop-MeetIrred}, we know that a degree
of enumerability must be meet-irreducible and either total (i.e., a degree of
solvability) or uncountable.  This begs the question of whether there are
Medvedev degrees that are closed, meet-irreducible, and uncountable, yet not
degrees of enumerability.  We show that the Medvedev degree of the $\{0,1\}
$-valued diagonally nonrecursive functions is such a degree.

Recall that $f \in \omega^\omega$ is \emph{diagonally nonrecursive} (DNR for
short) if $\forall e(\Phi_e(e)\da \imp f(e) \neq \Phi_e(e))$.  Let $\dnr_2 =
\{f \in 2^\omega : \text{$f$ is DNR}\}$.

\begin{Lemma}\label{lem-TreeEnum}
Let $T \subseteq \omega^{<\omega}$ be an infinite $h$-bounded tree for some
$h \in \omega^\omega$.  If $A \subseteq \omega$ is such that $\MP{E}_A \leqs
[T]$, then $A$ is r.e.\ in $T \oplus h$.
\end{Lemma}

\begin{proof}
Let $\Phi$ be such that $\Phi([T]) \subseteq \MP{E}_A$. Using $T \oplus h$ as
an oracle, enumerate the set
\begin{align*}
B = \{n : \exists k(\forall\,\text{$h$-bounded
      $\sigma$ with $|\sigma|=k$})(\sigma \in T \imp n
      \in \ran(\Phi(\sigma)))\}.
\end{align*}
We show that $B = A$, thus showing that $A$ is r.e.\ in $T \oplus h$.

Suppose that $n \in B$.  Let $k$ be such that $n \in \ran(\Phi(\sigma))$
whenever $\sigma \in T$ has length $k$.  Let $f \in [T]$. Then
$f\!\restriction\!k \in T$, so $n \in \ran(\Phi(f\!\restriction\!k))$.
However, $\ran(\Phi(f)) = A$ because $\Phi(f) \in \MP{E}_A$, so it must be
that $n \in A$.  Hence $B \subseteq A$.

Now suppose that $n \notin B$. Then for every $k$ there is
an $h$-bounded $\sigma$ of length $k$ with $\sigma \in T$ but $n \notin
\ran(\Phi(\sigma))$. So the subtree $S \subseteq T$ given by $S = \{\sigma
\in T : n \notin \ran(\Phi(\sigma))\}$ is infinite. By K\"{o}nig's lemma,
there is a path $f \in [S] \subseteq [T]$. However, $n \notin
\ran(\Phi(f))=A$, giving $n \notin A$ as desired.
\end{proof}

In the next proposition, our proof that $\degs(\dnr_2)$ is uncountable relies
on the following fact.  If $\MP{A} \subseteq 2^\omega$ is a nonempty
$\Pi^0_1$ class with no recursive member and $\MP B$ is a countable mass
problem with no recursive member, then $\MP B \nleqs \MP A$.  This fact
follows immediately from \cite{Jockusch- Soare:Pi01}*{Theorem~2.5}, which
essentially states that such an $\MP A$ must in fact have continuum-many
members that are all pairwise Turing incomparable and also all Turing
incomparable with all members of $\MP B$.  So in fact $\MP B \nleqw \MP A$,
where $\leqw$ is \emph{Muchnik reducibility}:  $\MP X \leqw \MP Y$ if
$(\forall g \in \MP{Y})(\exists f \in \MP{X})(f \leqT g)$.  That $\MP B
\nleqs \MP A$ can also be deduced from the well-known fact that the image of
a recursively bounded $\Pi^0_1$ class under a Turing functional is another
recursively bounded $\Pi^0_1$ class (see \cite{Simpson}*{Theorem~4.7}), which
is easier to prove than~\cite{Jockusch-Soare:Pi01}*{Theorem~2.5}.  Suppose
for a contradiction that $\MP B \leqs \MP A$ via the Turing functional
$\Phi$.  Then $\MP{B}_0 = \Phi(\MP A) \subseteq \MP B$ is a countable
recursively bounded $\Pi^0_1$ class and therefore must have a recursive
member, contradicting that $\MP B$ has no recursive member. This argument can
also be used to show that $\MP B \nleqw \MP A$ because if $\MP B \leqw \MP
A$, then (see \cite{Simpson}*{Lemma~6.9}) there is a nonempty $\Pi^0_1$ class
$\MP{A}_0 \subseteq \MP A$ such that $\MP B \leqs \MP{A}_0$, and then the
argument can be repeated with $\MP{A}_0$ in place of $\MP A$.

\begin{Proposition}\label{prop-ClosedNotEnum}
The Medvedev degree $\degs(\dnr_2)$ is closed, meet-irreducible, and
uncountable, yet also not a degree of enumerability (in fact, it does not
bound any nonzero degree of enumerability).
\end{Proposition}

\begin{proof}
It is well-known that $\dnr_2$ is a $\Pi^0_1$ class because $\dnr_2 = [T]$
for the recursive tree
\begin{align*}
T = \{\sigma \in 2^{<\omega} : (\forall e < |\sigma|)(\text{$\Phi_e(e)$
     halts within $|\sigma|$ steps} \imp \sigma(e) \neq \Phi_e(e))\}.
\end{align*}
By the above discussion, if $\MP B$ is a countable mass problem with no
recursive member, then $\MP B \nleqs \dnr_2$. Hence $\degs(\dnr_2)$ is
uncountable.  That $\degs(\dnr_2)$ is meet-irreducible follows from
Lemma~\ref{lem:uniform-meet-irr}, as it is easy to see that $\MP A = \dnr_2$
is uniform.  If $\sigma \subset f$ for an $f$ in $\dnr_2$, consider the
reduction procedure $\MP{A}_\sigma \leqs \MP{A}$ given by $\Phi(f)=\sigma \#
f$.

That $\degs(\dnr_2)$ is not a degree of enumerability follows from
Lemma~\ref{lem-TreeEnum}. We know that $\dnr_2 = [T]$ for a recursive tree $T
\subseteq 2^{<\omega}$.  Thus if $\MP{E}_A \leqs \dnr_2$ for some $A
\subseteq \omega$, then $A$ would have to be r.e.\ by Lemma~\ref{lem-TreeEnum}.
However, if $A$ is r.e., then $\MP{E}_A$ would have a recursive member, in
which case $\dnr_2 \nleqs \MP{E}_A$.  Thus there is no $A$ such that $\dnr_2
\equivs \MP{E}_A$.  In fact $\dnr_2$ does not bound any nonzero degree of
enumerability.
\end{proof}

If  $\md{E}_A\ne \mathbf{0}$ and $\md{E}_A$ is not quasiminimal, then there
are nonrecursive functions $f$ such that $\{f\} \leqs \MP{E}_{A}$, so
$\md{E}_A$ bounds some nonzero closed degree. As observed in
\cite{Lewis-Shore-Sorbi}, there are also quasiminimal degrees of
enumerability $\md{E}_A$ that bound nonzero closed degrees. Given an
infinite set $A$, consider the mass problem
\begin{align*}
\MP{C}_A = \{f : \text{$f$ is one-to-one and
$\ran(f) \subseteq A$}\}.
\end{align*}

As observed in~\cite{BianchiniSorbi}, $\MP{C}_A$ is closed, $\degs(\MP{C}_A)
\leqs \md{E}_A$, and, if $A$ is immune (meaning that $A$ has no infinite
r.e.\ subset), then $\degs(\MP{C}_A) \neq \md 0$.  So, if $A$ is immune and
of quasiminimal e-degree (which is the case, for instance, if $A$ is a
$1$-generic set, see \cite{Copestake:1-generic}), then we have a quasiminimal
degree of enumerability which bounds a nonzero closed degree. On the other
hand, if $A$ contains an infinite set $B$ such that $A \nleqe B$, then
$\md{E}_A \nleqs \degs(\MP{C}_A)$ because in this case $\MP{C}_A \leqs
\MP{E}_B$ (as $\MP{C}_B \subseteq \MP{C}_A$) but $\MP{E}_A \nleqs \MP{E}_B$.
This gives examples of sets $A$, even of total e-degree, for which
$\mathbf{0}<_\textrm{s} \degs(\MP{C}_A) <_\textrm{s} \md{E}_A$.

\begin{Proposition}
There is a total $f \colon \omega \imp \omega$ such that
$\degs{(\MP{C}_{\graph(f)})} \neq \mathbf{0}$
and $\MP{C}_{\graph(f)} <_\textrm{s} \MP{E}_{\graph(f)}$.
\end{Proposition}

\begin{proof}
By the above remarks and by Lemma~\ref{lem-EnumTotGraph}, consider two
biimmune sets $A,B$ with $A \midT B$, and let $f = \chi_A \oplus \chi_B$
(where $\chi_Z$ denotes the characteristic function of $Z$). Then $\graph(f)$
is immune, and it contains an infinite subset (for instance $\{\langle 2x,
f(2x)\rangle : x \in \omega\}$) to which it does not Turing-reduce, and
hence, by totality, to which it does not e-reduce.
\end{proof}

However, if $f$ is total, then $\MP{C}_{\graph(f)} \equivs
\MP{E}_{\graph(f)}$ is almost true, as argued in the following
proposition.

\begin{Proposition}
If $f \colon \omega \imp \omega$ is total, then there is a set $B \equive
\graph{(f)}$ such that $ \MP{C}_B \equivs \MP{E}_{\graph(f)}$.
\end{Proposition}

\begin{proof}
Given $f$ total, let $B=\{\sigma\in \omega^{<\omega} \colon \sigma \subset
f\}$. It is easy to see that $f \equive B$, so $\MP{E}_{\graph(f)} \equivs
\MP{E}_{B}$.  To see that $\MP{E}_{B} \leqs \MP{C}_{B}$, let $\Phi$ be a
Turing functional such that, for every $g$ and $n$, $\Phi(g)(n)$ searches for
an $m$ such that $g(m)$ is a string $ \sigma$ with $|\sigma| \geq n$ and then
outputs $\sigma\!\restriction\!n$.  Then $\ran(g)$ is an infinite subset of
$B$ whenever $g \in \MP{C}_B$, in which case $\ran(\Phi(g)) = B$.  Hence $
\Phi$ witnesses that $\MP{E}_{B} \leqs \MP{C}_{B}$.
\end{proof}

While it is true that every total degree of enumerability bounds (and in fact
is equivalent to) a closed mass problem, if we move away from totality,
then all possibilities may occur.  That is, there are nontotal (in fact
quasiminimal) degrees of enumerability that are closed (in fact compact, see
Theorem~\ref{thm:qm-compact} below), and there are nonzero
degrees of enumerability that do not bound nonzero closed degrees
(see Theorem~\ref{thm:no-bounding} below).

\subsection*{Compactness and cototality}
We make use of \emph{uniformly e-pointed trees}.
This notion was originally introduced by Montalb\'an~\cite{Montalban:Comm}
in the context of computable structure theory (see also~\cite{Montalban:Book}),
and it has since been studied by McCarthy in the context of
the enumeration degrees~\cite{Mccarthy}.
Montalb\'an's uniformly e-pointed trees are subtrees of
$2^{<\omega}$, which we refer to as \emph{uniformly e-pointed trees w.r.t.\
sets}.  We find it convenient to work with finitely branching subtrees of
$\omega^{<\omega}$ instead, so we define \emph{uniformly e-pointed trees
w.r.t.\ functions}.\footnote{The authors are thankful to Alexandra A.\
Soskova and Mariya I.\ Soskova for bringing to their attention, after a first
draft of this paper was completed, the notion of uniformly e-pointed tree
w.r.t.\ sets, called simply \emph{uniformly e-pointed} by Montalb\'an and
McCarthy.}

\begin{Definition}\label{def-plus}
For a function $g \in 2^\omega$, let $g^+ = \{n : g(n) = 1\}$ denote the set
of which $g$ is the characteristic function.
\end{Definition}

\begin{Definition}{\ }
\begin{itemize}
\item A \emph{uniformly e-pointed tree with respect to sets} is a tree $T
    \subseteq 2^{<\omega}$ with no leaves for which there is an enumeration
    operator $\Psi$ such that $(\forall g \in [T])(\Psi(g^+) = T)$.

\item A  \emph{uniformly e-pointed tree with respect to functions} is a
    finitely branching
    tree $T \subseteq \omega^{<\omega}$ with no leaves for which there is
    an enumeration operator $\Psi$ such that $(\forall g \in [T])(\Psi(g) =
    T)$.  (Recall that, as $\Psi$ is an enumeration operator, $\Psi(g)$
    means $\Psi(\graph(g))$.)
\end{itemize}
\end{Definition}

We show that the two notions of uniform e-pointedness coincide up to
e-equivalence.

\begin{Proposition}\label{prop:from-sets-to-functions}
Every uniformly e-pointed tree w.r.t.\ sets is a uniformly e-pointed tree
w.r.t.\ functions.
\end{Proposition}

\begin{proof}
Let $T \subseteq 2^{<\omega}$ be a uniformly e-pointed tree w.r.t.\ sets. Let
$\Psi$ be an enumeration operator such that $(\forall g \in [T])(\Psi(g^{+})
= T)$. Fix an enumeration operator $\Gamma$ such that $(\forall A \subseteq
\omega)(\Gamma(\chi_A) = A)$.  By composing $\Psi$ and $\Gamma$, we get an
enumeration operator $\Theta$ such that $(\forall g \in [T])(\Theta(g) = T)$.
Thus $T$ is a uniformly e-pointed tree w.r.t.\ functions.
\end{proof}

\begin{Proposition}\label{prop:link}
Let $T \subseteq \omega^{<\omega}$ be a uniformly e-pointed tree w.r.t.\
functions.  Then there is a uniformly e-pointed tree $S \subseteq
2^{<\omega}$ w.r.t.\ sets such that $S \equive T$. (In fact we may choose $S$
so that $[S]$ consists of exactly the characteristic functions of the graphs
of elements of $[T]$.)
\end{Proposition}

\begin{proof}
Let $T \subseteq \omega^{<\omega}$ be a uniformly e-pointed tree w.r.t.\
functions.  Say that $ \gamma \in 2^{<\omega}$ is \emph{consistent with $T$}
if there is a $\sigma \in T$ such that
\begin{align*}
(\forall \la i, n \ra < |\gamma|)(i < |\sigma| \andd
(\gamma(\la i,n \ra) = 1 \biimp \sigma(i)=n)).
\end{align*}
Notice that if $\eta \subseteq \gamma \in 2^{<\omega}$ and $\gamma$ is
consistent with $T$, then $\eta$ is also consistent with $T$.  Let
\begin{align*}
S = \{\gamma \in 2^{<\omega} : \text{$\gamma$ is consistent with $T$}\}.
\end{align*}
Then $S$ is a tree, $S$ has no leaves because $T$ has no leaves, and it is
immediate to check that $S \leqe T$. To see that $T \leqe S$, observe that
\begin{align*}
T = \{\sigma \in \omega^{<\omega} : (\exists \gamma \in S)(\forall i <
|\sigma|)(\la i, \sigma(i) \ra \in \dom(\gamma) \andd
\gamma(\la i, \sigma(i) \ra) = 1)\}.
\end{align*}
Furthermore, $[S] = \{\chi_{\graph(f)} : f \in [T]\}$.  If $f \in [T]$, then
$\chi_{\graph(f)}\!\restriction\!n$ is consistent with $T$ for every $n$ (as
witnessed by $f\!\restriction\!n$), thus $\chi_{\graph(f)}\!\restriction\!n
\in S$ for every $n$, thus $\chi_{\graph(f)} \in [S]$. Conversely, suppose
that $f \notin [T]$.  Then there is an $n$ such that $f\!\restriction\!n
\notin T$. We want to find an $m$ such that
$\chi_{\graph(f)}\!\restriction\!m \notin S$ in order to conclude that
$\chi_{\graph(f)} \notin [S]$.  By the fact that $T$ is finitely branching,
let $k$ be large enough so that $(\forall i < | \sigma|)(\sigma(i) < k)$
whenever $\sigma \in T$ has length $\leq n$.  Let $m > \la n, k \ra$. Suppose
for a contradiction that $\chi_{\graph(f)}\!\restriction\!m$ is consistent
with $T$, and let $ \sigma$ witness this. Then it must be that $|\sigma| \geq
n$ and $(\forall i < n)(\sigma(i) = f(i))$. Thus $\sigma \supseteq
f\!\restriction\!n$, contradicting that $f\!\restriction\!n \notin T$.  Thus
$ \chi_{\graph(f)}\!\restriction\!m$ is not consistent with $T$, so
$\chi_{\graph(f)}\!\restriction\!m \notin S$.

To finish, we need to find an enumeration operator $\Psi$ such that $(\forall
g \in [S])(\Psi(g^{+}) = S)$.  So let $\Theta$ be an enumeration operator
such that $ (\forall f \in [T])(\Theta(f) = T)$, and let $\Gamma$ be an
enumeration operator witnessing that $S \leqe T$.  By composing $\Gamma$ and
$\Theta$, we get an enumeration operator $\Psi$ such that $(\forall f \in
[T])(\Psi(f) = S)$. However, this is exactly what we want because we have
shown that if $g \in [S]$, then $g = \chi_{\graph(f)}$ for some $f \in [T]$
and therefore that if $g \in [S]$ then $g^+ = \graph(f)$ for some $f \in
[T]$. Thus $(\forall g \in [S]) (\Psi(g^+) = S)$, as desired (recall that
$\Theta(f)=\Theta(\graph(f))$).
\end{proof}

A set $A$ is called \emph{cototal} if $A \leqe \overline{A}$, and an e-degree
is called \emph{cototal} if it contains a cototal set~\cite{Andrews-et-al}.
Every uniformly e-pointed tree w.r.t.\ sets is cototal by
\cite{Mccarthy}*{Theorem~4.7}, and, by \cite{Mccarthy}*{Corollary~4.9.1}, an
e-degree is cototal if and only if it contains a uniformly e-pointed tree
w.r.t.\ sets.

\begin{Proposition}\label{prop-ePointedCototal}
An enumeration degree is cototal if and only if it contains a uniformly
e-pointed tree w.r.t.\ functions.
\end{Proposition}

\begin{proof}
An e-degree is cototal if and only if it it contains a uniformly e-pointed
tree w.r.t.\ sets by \cite{Mccarthy}*{Theorem~4.7} if and only if it contains
a uniformly e-pointed tree w.r.t.\ functions by
Proposition~\ref{prop:from-sets-to-functions} and
Proposition~\ref{prop:link}.

We also find it interesting to give a more direct proof that every uniformly
e-pointed tree w.r.t.\ functions has cototal enumeration degree. This can be
accomplished via the easy characterization of the cototal enumeration degrees
in terms of the \emph{skip operator} from \cite{Andrews-et-al}.

Recall that $(\Psi_e : e \in \omega)$ is a standard list of all enumeration
operators, and recall the following definitions.

\begin{itemize}
\item For an $A \subseteq \omega$, $K_A = \{\la e, x \ra : x \in
    \Psi_e(A)\}$.

\item For an $A \subseteq \omega$, $A^\skp = \overline{K_A}$ is
    called the \emph{skip} of $A$.
\end{itemize}
By \cite{Andrews-et-al}*{Proposition~1.1},
a set $A \subseteq \omega$ has cototal enumeration degree if and only if $A \leqe
A^\skp$.

Let $T$ be a uniformly e-pointed tree w.r.t.\ functions.  We show that $T
\leqe T^\skp$ and therefore that $T$ has cototal enumeration degree. Let
$\Psi$ be an enumeration operator such that $(\forall f \in [T])(\Psi(f) =
T)$.  For each $n \in \omega$, let $T^n = \{\sigma \in T : |\sigma| = n\}$
denote level $n$ of $T$.  For $b, n \in \omega$, let $b^n = \{\sigma \in
\omega^{<\omega} : |\sigma| = n \andd (\forall i < |\sigma|)(\sigma(i) <
b)\}$ denote the set of all $b$-bounded strings of length $n$. Let $B = \{\la
n, b \ra : T^n \smallsetminus b^n \neq \emptyset\}$. That is, $B$ is the set
of all pairs $\la n, b \ra$ where $b$ is not big enough to bound every entry
of every string in $T^n$.  We have $B \leqe T$, thus $B\le_1 K_T$, and
therefore $\overline B \leqe T^\skp$.  The point is that if $\la n, b \ra \in
\overline B$, then $T^n \subseteq b^n$, which allows us enumerate $T$ from an
enumeration of $\overline{T} \oplus \overline{B}$. Indeed,
\begin{align*}
T = \{\sigma : (\exists \la n, b \ra \in \overline{B})(\exists L
\subseteq b^n \cap \overline{T})(\forall \tau \in b^n
\smallsetminus L)(\sigma \in \Psi(\tau))\}.
\end{align*}
That is, we know that $\sigma \in T$ when we see a bound $T^n \subseteq b^n$
and a set of strings $L \subseteq b^n$ that are \emph{not} in $T$ such that
the remaining $\tau \in b^n \smallsetminus L$ all satisfy $\sigma \in
\Psi(\tau)$. Thus $T \leqe \overline{T} \oplus \overline{B} \leqe T^\skp$, so
$T$ has cototal enumeration degree.
\end{proof}

We extend the cototal terminology to the degrees of enumerability by saying
that $\md{E}_{A}$ is \emph{cototal} if $A$ has cototal enumeration degree. To
conclude this section, we show that cototality and compactness are equivalent
properties of a degree of enumerability.

\begin{Lemma}\label{lem-EnumTree}
Let $A \subseteq \omega$ be nonempty, and let $\MP C \subseteq \omega^\omega$
be closed such that $\MP C \leqs \MP{E}_A$.  Then there is a tree $T
\subseteq \omega^{<\omega}$ with no leaves such that $T \leqe A$ and $[T]
\subseteq \MP C$.  Furthermore, if $\MP C$ is compact, then $T$ is finitely
branching.
\end{Lemma}

\begin{proof}
Let $\Phi$ be a Turing functional such that $\Phi(\MP{E}_A)
\subseteq \MP C$,
with $\MP{C}$ closed. Let
\begin{align*}
T = \{\sigma : \exists \alpha(\ran(\alpha) \subseteq A \andd
\sigma \subseteq \Phi(\alpha))\}.
\end{align*}
Then $T$ is a tree and $T \leqe A$. To see that $T$ has no leaves, let
$\sigma \in T$, and let $\alpha$ be such that $\ran(\alpha) \subseteq A$ and
$\sigma \subseteq \Phi(\alpha)$. Let $f \colon \omega \imp \omega$ be such
that $\alpha \subset f$ and $\ran(f) = A$.  Let $\beta$ be such that $\alpha
\subseteq \beta \subset f$ and $\Phi(\beta)(|\sigma|)\da$.  Then $\sigma
\subsetneq \Phi(\beta) \in T$, so $\sigma$ is not a leaf.  To see that $[T]
\subseteq \MP C$, we consider a $g \in [T]$ and show that $g$ is in the
closure of $\MP C$. To this end, let $n \in \omega$, let $\alpha$ be such
that $\ran(\alpha) \subseteq A$ and $g\!\restriction\!n \subseteq
\Phi(\alpha)$, and let $f \colon \omega \imp \omega$ be such that $\alpha
\subset f$ and $\ran(f) = A$. Then $\Phi(f) \in \MP{C}$ and
$\Phi(f)\!\restriction\!n = g\!\restriction\!n$. Hence $g$ is in the closure
of $\MP C$, so $g \in \MP C$.

Lastly, if $\MP C$ is compact, then $[T]$ is compact because $[T] \subseteq
\MP C$.  This means that $T$ must be finitely branching because $T$ has no
leaves.
\end{proof}

\begin{Lemma}\label{lem-EnumCompact}
Let $A \subseteq \omega$ be nonempty. Then $\md{E}_A$ is compact if and only
if there is a uniformly e-pointed tree $T \subseteq
\omega^{<\omega}$ w.r.t.\ functions such that $T \equive A$.
\end{Lemma}

\begin{proof}
Suppose that $\MP{E}_A \equivs \MP C$, where $\MP C$ is compact.  Let $\Phi$
be a Turing functional such that $\Phi(\MP C) \subseteq \MP{E}_A$.  $\MP C$
is compact, so its image $\MP D = \Phi(\MP C)$ is also compact by the
continuity of the Turing functional $\Phi$. By Lemma~\ref{lem-EnumTree},
there is a finitely branching tree $T \subseteq \omega^{<\omega}$ with no
leaves such that $T \leqe A$ and $[T] \subseteq \MP D \subseteq \MP{E}_A$.
Furthermore, $A = \bigcup_{\sigma \in T} \ran(\sigma)$ because $[T] \subseteq
\MP{E}_A$, which implies that $A \leqe T$.  Hence $T \equive A$. Also, if
$g \in [T]$, then $\ran(g)=A$, thus there is a uniform procedure enumerating
$A$ and hence $T$ from any enumeration of $g$, which shows that $T$ is
uniformly e-pointed w.r.t.\ functions.

Conversely, suppose that there is a uniformly e-pointed
tree $T \subseteq \omega^{<\omega}$ w.r.t.\ functions such that $T \equive A$.
Then $\MP{E}_A \leqs [T]$, as one can uniformly transform
any function $g \in [T]$ into a function that enumerates $A$ because $T\leqe
g$ uniformly and $A\leqe T$. To see that $[T] \leqs \MP{E}_A$, consider the
Turing functional $\Phi$ which, on an $f \in \MP{E}_A$, uses $\ran(f)$ to
simultaneously enumerate $T$ (via the reduction $T \leqe A$) and a path
through $T$ (which is possible because $T$ has no leaves).  Thus $\MP{E}_A
\equivs [T]$, and $[T]$ is compact because $T$ is finitely branching.
\end{proof}

Observe that the proof  of Lemma~\ref{lem-EnumCompact} also proves the
following fact, which we record for posterity.

\begin{Proposition}
Let $A \subseteq \omega$ be nonempty.  If $T \subseteq \omega^{<\omega}$
is a finitely branching tree with no leaves such that $T \leqe A$ and $[T]
\subseteq \MP{E}_A$, then $T \equive A$, $[T] \equivs \MP{E}_A$,
and $T$ is uniformly e-pointed w.r.t.\ functions.
\end{Proposition}

\begin{Theorem}\label{thm:cototal=compact}
$\md{E}_A$ is a compact degree of enumerability if and only if $A$ has
cototal enumeration degree.  Hence a degree of enumerability is compact if
and only if it is cototal.
\end{Theorem}

\begin{proof}
The degree of enumerability $\md{E}_A$ is compact if and only if $A \equive
T$ for some uniformly e-pointed tree $T \subseteq \omega^{<\omega}$ w.r.t.\
functions by Lemma~\ref{lem-EnumCompact}, which is the case if and only if
$A$ has cototal enumeration degree by Proposition~\ref{prop-ePointedCototal}.
\end{proof}

\subsection*{A quasiminimal degree of enumerability that is compact.}
The existence of quasiminimal problems of enumerability that are equivalent
to compact mass problems is a consequence of
Theorem~\ref{thm:cototal=compact} and the fact that there are cototal
quasiminimal e-degrees~\cite{Andrews-et-al}. We think, however, that it is
instructive to directly construct a quasiminimal uniformly e-pointed tree
w.r.t.\ functions.  The corresponding degree of enumerability is then
quasiminimal by definition and compact by Lemma~\ref{lem-EnumCompact}.

Recall that $\la \cdot, \cdot \ra \colon \omega^2 \imp \omega$ is the
recursive pairing function.  Let $\pi_0, \pi_1 \colon \omega \imp \omega$
denote the projection functions $\pi_0(\la m, n \ra) = m$ and $\pi_1(\la m, n
\ra) = n$.

\begin{Lemma}\label{lem-SelfEnumTree}
There is a finitely branching tree $A \subseteq \omega^{<\omega}$ such that
\begin{itemize}
\item $A$ has no leaves,
\item $A$ is quasiminimal, and
\item $\ran(\pi_1 \circ f) = A$ for every $f \in [A]$.
\end{itemize}
Notice that such a tree is uniformly e-pointed w.r.t.\ functions.
\end{Lemma}

\begin{proof}
For the purposes of this proof, we make the following definitions for finite
trees $T, S \subseteq \omega^{<\omega}$:
\begin{itemize}
\item $\leaves(T) = \{\sigma \in T : \text{$\sigma$ is a leaf of $T$}\}$.
\item $S$ \emph{leaf-extends} $T$ if $T \subseteq S$ and $(\forall \tau \in
    S \smallsetminus T)(\exists \sigma \in \leaves(T))(\sigma \subseteq
    \tau)$;

\item $S$ \emph{properly leaf-extends} $T$ if $S$ leaf-extends $T$ and
    $(\forall \sigma \in T)(\exists \tau \in S)(\sigma \subsetneq \tau)$.
\end{itemize}

We build a sequence of finite trees $A_0 \subseteq A_1 \subseteq A_2
\subseteq \dots$, where $A_{s+1}$ properly leaf-extends $A_s$ for each $s \in
\omega$. This way, $A = \bigcup_{s \in \omega} A_s$ has no leaves and is
finitely branching. Furthermore, we build the sequence so that
\begin{align*}
(\forall s
\in \omega)(\forall \sigma \in \leaves(A_{s+1}))(A_s \subseteq \ran(\pi_1
\circ \sigma) \subseteq A_{s+1}).
\end{align*}
This ensures that $\ran(\pi_1 \circ f) = A$ for every $f \in [A]$. To help
ensure that $A \gee \emptyset$, we also maintain a sequence of finite sets of
strings $O_0 \subseteq O_1 \subseteq O_2 \subseteq \dots$ such that $\forall
s(A_s \cap O_s = \emptyset)$.

We satisfy the requirements
\begin{align*}
\MP{Q}_e\colon& A \neq W_e\\
\MP{R}_e\colon& \text{if $\Psi_e(A)$ is the graph of a total
                            function $f$, then $f$ is recursive.}
\end{align*}

\medskip

\noindent \emph{Stage $0$}: set $A_0 = \{\emptyset\}$, and set $O_0 =
\emptyset$.

\medskip

\noindent \emph{Stage $s+1 = 2e+1$}: We satisfy $\MP{Q}_e$.  If $W_e$ is
finite, then set $O_{s+1} = O_s$.  If $W_e$ is infinite, then choose any
$\sigma \in W_e \smallsetminus A_s$, and set $O_{s+1} = O_s \cup \{\sigma\}$.
To extend $A_s$ to $A_{s+1}$, first choose $n$ greater than (the code of)
every element in $O_{s+1}$.  Then choose any enumeration $(\alpha_i)_{i < k}$
of $A_s$. Then let $\beta$ be the string $\la \la n, \alpha_0 \ra, \la n,
\alpha_1 \ra, \dots, \la n, \alpha_{k-1} \ra \ra$.  Now let $A_{s+1}$ be the
tree obtained by extending each leaf of $A_s$ by $\beta$:
\begin{align*}
A_{s+1} = \{\sigma : (\exists \tau \in \leaves(A_s))(\sigma \subseteq
\tau^\smf\beta)\}.
\end{align*}
Having chosen $n$ big enough, we have guaranteed that $A_{s+1}$ is disjoint
from $O_{s+1}$.

\medskip

\noindent \emph{Stage $s+1 = 2e+2$}: We satisfy $\MP{R}_e$. Set $O_{s+1} =
O_s$.  For finite trees $T, S \subseteq \omega^{<\omega}$, call $S$ a
\emph{good} extension of $T$ if $S$ leaf-extends $T$, $S \cap O_{s+1} =
\emptyset$, and $(\forall \sigma \in S)(\ran(\pi_1 \circ \sigma) \subseteq
S)$. Ask if there is a good extension $R$ of $A_s$ such that
\begin{align*}
(\exists m, n, o)(n \neq o \andd \la m, n \ra \in \Psi_e(R) \andd \la m, o \ra
\in \Psi_e(R)).
\end{align*}
If there is such an $R$, let $\widehat{A}_s = R$.  Otherwise, let
$\widehat{A}_s = A_s$. Now extend $\widehat{A}_s$ to $A_{s+1}$ the same way
that we extend $A_s$ to $A_{s+1}$ during the odd stages.  The fact that
$(\forall \sigma \in \widehat{A}_s)(\ran(\pi_1 \circ \sigma) \subseteq
\widehat{A}_s)$ ensures that $(\forall \sigma \in A_{s+1})(\ran(\pi_1 \circ
\sigma) \subseteq A_{s+1})$.  This completes the construction.

\medskip

Let $A = \bigcup_{s \in \omega} A_s$. We show that all requirements are
satisfied.

For requirement $\MP{Q}_e$, consider stage $s+1 = 2e+1$.  If $W_e$ is finite,
then $A \neq W_e$ because $A$ is infinite.  If $W_e$ is infinite, then at
stage $s+1$ we chose a $\sigma \in W_e \smallsetminus A_s$ and put $\sigma$ in
$O_{s+1}$.  Thus $\forall t (\sigma \notin A_t)$, so $\sigma \notin A$. Hence
$A \neq W_e$.

For requirement $\MP{R}_e$, suppose that $\Psi_e(A)$ is the graph of a total
function $f$, and consider stage $s+1 = 2e+2$.  We show that $\graph(f)$ is
r.e., which implies that $f$ is recursive. Let
\begin{align*}
X = \{\la m, n \ra : \text{there is a good extension $B$ of
        $A_s$ with $\la m, n \ra \in \Psi_e(B)$}\}
\end{align*}
(where here `good' means with respect to the $O_{s+1}$ at stage $s+1$).
Clearly $X$ is r.e.  We show that $X = \graph(f)$.  For $\graph(f) \subseteq
X$, suppose that $\la m, n \ra \in \graph(f) = \Psi_e(A)$.  Let $t \geq s+1$
be such that $\la m, n \ra \in \Psi_e(A_t)$.  Then $A_t$ is a good extension
of $A_s$ with $\la m, n \ra \in \Psi_e(A_t)$, so $\la m, n \ra \in X$.
Conversely, suppose that $\la m, n \ra \in X$, and let $B$ be a good
extension of $A_s$ with $\la m, n \ra \in \Psi_e(B)$.  If $\la m, n \ra
\notin \graph(f)$, then $\la m, o \ra \in \graph(f) = \Phi_e(A)$, where $o =
f(m) \neq n$.  Let $t \geq s+1$ be such that $\la m, o \ra \in \Psi_e(A_t)$.
Then $A_t$ is a good extension of $A_s$, and, moreover, $A_t \cup B$ is also
a good extension of $A_s$.  Thus there is a good extension $R = A_t \cup B$
of $A_s$ such that $n \neq o \andd \la m, n \ra \in \Psi_e(R) \andd \la m, o
\ra \in \Psi_e(R)$, for some  $m, n, o \in \omega$.  Therefore, at stage
$s+1$, we extended $A_s$ to an $A_{s+1}$ such that
\begin{align*}
(\exists m, n, o)(n \neq o \andd
\la m, n \ra \in \Psi_e(A_{s+1}) \andd \la m, o \ra \in \Psi_e(A_{s+1})).
\end{align*}
This contradicts that $\Psi_e(A)$ is the graph of a function.

All together, we have that $A$ has no leaves, that $\ran(\pi_1 \circ f) = A$
for every $f \in [A]$ by construction, and that $A$ is quasiminimal by the
$\MP{Q}_e$ requirements and the $\MP{R}_e$ requirements.
\end{proof}

\begin{Theorem}\label{thm:qm-compact}
There is a degree of enumerability $\md{E}_A$ that is both quasiminimal and
compact. Hence $\md{E}_A$ is closed, nonzero, and does not bound any nonzero
degree of solvability.
\end{Theorem}

\begin{proof}
Let $A$ be the tree from Lemma~\ref{lem-SelfEnumTree}. Then $A$ has
quasiminimal e-degree, so $\md{E}_A$ is quasiminimal by definition.
Furthermore, $A$ is uniformly e-pointed w.r.t.\ functions, so $\md{E}_A$
is compact by Lemma~\ref{lem-EnumCompact}.
\end{proof}

\begin{Remark}
In Lemma~\ref{lem-SelfEnumTree}, one can make the tree $A$ be not cototal by
a small modification to the proof.  Thus although every uniformly e-pointed
tree w.r.t.\ functions has cototal \emph{e-degree} by
Proposition~\ref{prop-ePointedCototal}, it is not the case that every such
tree is cototal as a set.

To modify the proof, replace each old $\MP{Q}_e$ requirement $A \neq W_e$
with the new requirement $A \neq \Psi_{e}(\overline{A})$.  (Notice that a set
$A$ satisfying all of the new $\MP{Q}_e$ requirements is still not r.e.,
which is required in order for a set $A$ to be quasiminimal.)  To satisfy the
new $\MP{Q}_e$, modify stage $s+1 = 2e+1$ as follows.  If there are a finite
set $D \subseteq \overline{A_s}$ and a string $\sigma \in \overline{A_s}$
with $\sigma \in \Psi_e(D)$, then choose such a $D$ and $\sigma$, and set
$O_{s+1} = O_s \cup D \cup \{\sigma\}$.  Otherwise simply set $O_{s+1} =
O_s$. Then choose $n$ greater than (the code of) every element in $O_{s+1}$,
and extend $A_s$ to $A_{s+1}$ as before.  To verify that $\MP{Q}_e$ is
satisfied, suppose for a contradiction that $\Psi_e(\overline{A}) = A$.  As
$A$ is infinite and $A_s$ is finite, fix some $ \sigma \in A \setminus A_s$.
Let $D \subseteq \overline{A} \subseteq \overline{A_s}$ be a finite set such
that $\sigma \in \Phi_e(D)$.  Then, at stage $s+1$, we were able to choose a
$D$ and $\sigma$, ensuring that $\Phi_e(\overline{A}) \neq A$. This is a
contradiction.
\end{Remark}

\subsection*{A degree of enumerability that does not bound any nonzero
closed degree}

Finally, we show that there are examples of nonzero degrees of enumerability
that do not bound nonzero closed degrees. Such examples are of course
quasiminimal, and indeed the property of being nonzero but not above any
nonzero closed degree can be viewed as an interesting generalization of
quasiminimality. Theorem~\ref{thm:no-bounding} below can also be phrased by
saying that there are nonzero degrees of enumerability that do not lie in the
filter generated by the nonzero closed degrees, which coincides with the
collection of all Medvedev degrees bounding nonzero closed degrees
(see~\cite{Sorbi-filters}).

\begin{Lemma}\label{lem-Prob9Helper}
There is a set $A \gee \emptyset$ such that, for all $T \leqe A$, if $T$ is a
subtree of $\omega^{<\omega}$ with no leaves, then $T$ has an r.e.\ subtree
with no leaves.
\end{Lemma}

\begin{proof}
For the purposes of this proof, we assume that if $\Psi$ is an enumeration
operator, $X \subseteq \omega$, and $\Psi(X)$ enumerates some $\sigma \in
\omega^{<\omega}$ (i.e., $\sigma \in \Psi(X)$), then it also enumerates all
$\tau \subseteq \sigma$.   In fact, from any enumeration operator $\Gamma$,
one can effectively produce an enumeration operator $\Psi$ such that, for all
$X$,
\begin{itemize}
\item $\Psi(X)$ is a tree, and
\item if $\Gamma(X)$ is a tree, then $\Psi(X) = \Gamma(X)$.
\end{itemize}
To accomplish this, just take
$\Psi=\{\la \tau, D \ra : (\exists \sigma)(\tau \subseteq \sigma
\andd \la \sigma , D \ra \in \Gamma)\}$.  Therefore, we can define an
effective list $(\Psi_e : e \in \omega)$ of enumeration operators such that
\begin{itemize}
\item $\Psi_e(X)$ is a tree for every $e$ and $X$, and
\item if $T \leqe X$ for a tree $T$ and set a $X$, then there is an $e$
    such that $\Psi_e(X) = T$.
\end{itemize}
Also, recall the notation $g^+ = \{n : g(n) = 1\}$ from
Definition~\ref{def-plus}. We extend this notation to strings $\alpha \in
2^{< \omega}$ by defining $\alpha^+ = \{i < |\alpha| : \alpha(i)=1\}$.
Additionally, if $A \subseteq \omega$ and $\alpha \in 2^{<\omega}$, we write
$A \subseteq^+ \alpha^+$ to mean that $(\forall n < |\alpha|)(n \in A \imp
\alpha(n) = 1)$ (i.e., $\{n \in A: n < |\alpha|\} \subseteq \alpha^+$).

We satisfy the requirements
\begin{align*}
\MP{Q}_e\colon& A \neq W_e\\
\MP{R}_e\colon& \text{either $\Psi_e(A)$ contains a
leaf or there is an r.e.\ $T \subseteq \Psi_e(A)$ with no leaves.}
\end{align*}

We build a sequence of binary strings $\alpha_0 \subseteq \alpha_1 \subseteq
\alpha_2 \subseteq \dots$ along with sequences of recursive sets $I_0
\supseteq I_1 \supseteq I_2 \supseteq \dots$ and $J_0 \subseteq J_1 \subseteq
J_2 \subseteq \dots$ such that, for every $s \in \omega$, $I_s \smallsetminus
J_s$ is infinite and $J_s \subseteq^+ \alpha_s^+ \subseteq I_s$.  In the end,
we let $A = \bigcup_{s \in \omega}\alpha_s^+$, and we have $\bigcup_{s \in
\omega} J_s \subseteq A \subseteq \bigcap_{s \in \omega} I_s$.

\medskip

\noindent \emph{Stage $0$}: Set $\alpha_0 = \emptyset$, set $I_0 = \omega$,
and set $J_0 = \emptyset$.

\medskip

\noindent \emph{Stage $s+1 = 2e+1$}:  We satisfy $\MP{Q}_e$.  Let $n \in I_s
\smallsetminus J_s$ be least such that $n > |\alpha_s|$. If $n \in W_e$, set
$I_{s+1} = I_s \smallsetminus \{n\}$, set $J_{s+1} = J_s$, and extend
$\alpha_s$ to an $\alpha_{s+1}$ with $J_s \subseteq^+ \alpha_{s+1}^+
\subseteq I_s$ and $\alpha_{s+1}(n) = 0$.  If $n \notin W_e$, set $I_{s+1} =
I_s$, set $J_{s+1} = J_s$, and extend $\alpha_s$ to an $\alpha_{s+1}$ with
$J_s \subseteq^+ \alpha_{s+1}^+ \subseteq I_s$ and $\alpha_{s+1}(n) = 1$.

\medskip

\noindent \emph{Stage $s+1 = 2e+2$}: We satisfy $\MP{R}_e$.  Ask if there is
a $\beta \supseteq \alpha_s$ and a recursive set $R$ such that
\begin{itemize}
\item $J_s \subseteq^+ \beta^+ \subseteq R \subseteq I_s$,
\item $R \smallsetminus J_s$ is infinite, and
\item there is a $\sigma \in \Psi_e(\beta^+)$ that is a leaf in
    $\Psi_e(R)$.
\end{itemize}
If there are such $\beta$ and $R$, set $\alpha_{s+1} = \beta$, set $I_{s+1} =
R$, and set $J_{s+1} = J_s$.  If there are no such $\beta$ and $R$, then set
$\alpha_{s+1} = \alpha_s$, set $I_{s+1} = I_s$, and choose any recursive
$J_{s+1}$ whose characteristic function extends $\alpha_{s+1}$, $J_s
\subseteq J_{s+1} \subseteq I_{s+1}$, and $J_{s+1} \smallsetminus J_s$ and
$I_{s+1} \smallsetminus J_{s+1}$ are both infinite. This completes the
construction.

\medskip

Let $A = \bigcup_s \alpha_s^+$.  The $\MP{Q}_e$ requirements are clearly
satisfied, and together they ensure that $A$ is not r.e.  Hence $A \gee
\emptyset$.

Now suppose that $T \leqe A$ is a tree with no leaves, and let $\Psi_e$ be
such that $T = \Psi_e(A)$.  At stage $s+1 = 2e+2$, there must not have been a
$\beta$ and an $R$ because if there were, then we would have $\beta =
\alpha_{s+1}$ and $\beta^+ \subseteq A \subseteq R = I_{s+1}$, so there would
be a leaf $\sigma \in \Psi_e(A) = T$. It must therefore be that
$\Psi_e(J_{s+1})$ is a tree with no leaves. To see this, suppose instead that
$\Psi_e(J_{s+1})$ has a leaf $\sigma$. Let $\beta$ be such that $\alpha_{s+1}
\subseteq \beta$, $\beta^+ \subseteq J_{s+1}$, and $\sigma \in
\Psi_e(\beta^+)$. Then at stage $s+1$, we could have taken $\beta$ and $R =
J_{s+1}$, which is a contradiction. This finishes the proof because
$\Psi_e(J_{s+1}) \subseteq T$ since $J_{s+1} \subseteq A$, and
$\Psi_e(J_{s+1})$ is r.e.\ since $J_{s+1}$ is recursive.
\end{proof}

\begin{Theorem}\label{thm:no-bounding}
There is a nonzero degree of enumerability that does not bound a nonzero
closed degree.
\end{Theorem}

\begin{proof}
Let $A$ be as in Lemma~\ref{lem-Prob9Helper}.  Consider a closed $\MP C \leqs
\MP{E}_A$.  By Lemma~\ref{lem-EnumTree}, there is a tree $T \leqe A$ with no
leaves such that $[T] \subseteq \MP C$.  By Lemma~\ref{lem-Prob9Helper}, $T$
has an r.e.\ subtree $S$ with no leaves. Thus $[S] \subseteq [T] \subseteq
\MP C$.  However, being a tree with no leaf, $S$ has a recursive path, so $\MP C$
has a recursive member, so $\degs(\MP C) = \md 0$.
\end{proof}

\section*{Acknowledgements}
We thank Douglas Cenzer, Antonio Montalb\'an, Alexandra A.\ Soskova, and
Mariya I.\ Soskova for helpful comments and discussions.

\bibliographystyle{amsplain}
\bibliography{ShaferSorbiEnumVsClosed}

\begin{bibdiv}
\begin{biblist}

\bib{Andrews-et-al}{unpublished}{
      author={Andrews, Uri},
      author={Ganchev, Hristo},
      author={Kuyper, Rutger},
      author={Steffen, Lempp},
      author={Miller, Joseph~S.},
      author={Soskova, Alexandra~A.},
      author={Soskova, Mariya~I.},
       title={On cototality and the skip operator in the enumeration degrees},
        note={preprint},
}

\bib{BalbesDwinger}{book}{
      author={Balbes, Raymond},
      author={Dwinger, Philip},
       title={Distributive lattices},
   publisher={University of Missouri Press, Columbia, Mo.},
        date={1974},
      review={\MR{0373985}},
}

\bib{BianchiniSorbi}{article}{
      author={Bianchini, Caterina},
      author={Sorbi, Andrea},
       title={A note on closed degrees of difficulty of the {M}edvedev
  lattice},
        date={1996},
        ISSN={0942-5616},
     journal={Mathematical Logic Quarterly},
      volume={42},
      number={1},
       pages={127\ndash 133},
         url={https://doi.org/10.1002/malq.19960420111},
      review={\MR{1373134}},
}

\bib{Copestake:1-generic}{article}{
      author={Copestake, Kate},
       title={{$1$}-genericity in the enumeration degrees},
        date={1988},
        ISSN={0022-4812},
     journal={Journal of Symbolic Logic},
      volume={53},
      number={3},
       pages={878\ndash 887},
         url={https://doi.org/10.2307/2274578},
      review={\MR{961005}},
}

\bib{Dyment1976Rus}{article}{
      author={Dyment, Elena~Z.},
       title={On some properties of the {M}edvedev lattice},
        date={1976},
     journal={Mathematics of the USSR. Sbornik},
      volume={101(143)},
      number={3},
       pages={360\ndash 379, 455},
      review={\MR{0432433}},
}

\bib{HinmanSurvey}{article}{
      author={Hinman, Peter~G.},
       title={A survey of {M}u\v cnik and {M}edvedev degrees},
        date={2012},
        ISSN={1079-8986},
     journal={Bulletin of Symbolic Logic},
      volume={18},
      number={2},
       pages={161\ndash 229},
         url={https://doi.org/10.2178/bsl/1333560805},
      review={\MR{2931672}},
}

\bib{Jockusch-Soare:Pi01}{article}{
      author={Jockusch, Carl~G., Jr.},
      author={Soare, Robert~I.},
       title={{$\Pi ^{0}_{1}$} classes and degrees of theories},
        date={1972},
        ISSN={0002-9947},
     journal={Transactions of the American Mathematical Society},
      volume={173},
       pages={33\ndash 56},
         url={https://doi.org/10.2307/1996261},
      review={\MR{0316227}},
}

\bib{Kuyper-Medvedev}{article}{
      author={Kuyper, Rutger},
       title={Natural factors of the {M}edvedev lattice capturing {IPC}},
        date={2014},
        ISSN={0933-5846},
     journal={Archive for Mathematical Logic},
      volume={53},
      number={7-8},
       pages={865\ndash 879},
         url={https://doi.org/10.1007/s00153-014-0393-8},
      review={\MR{3271365}},
}

\bib{Lewis-Shore-Sorbi}{article}{
      author={Lewis, Andrew E.~M.},
      author={Shore, Richard~A.},
      author={Sorbi, Andrea},
       title={Topological aspects of the {M}edvedev lattice},
        date={2011},
        ISSN={0933-5846},
     journal={Archive for Mathematical Logic},
      volume={50},
      number={3-4},
       pages={319\ndash 340},
         url={https://doi.org/10.1007/s00153-010-0215-6},
      review={\MR{2786756}},
}

\bib{Mccarthy}{unpublished}{
      author={McCarthy, Ethan},
       title={Cototal enumeration degrees and their applications to effective
  mathematics},
        date={2017},
        note={To appear in Proceedings of the American Mathematical Society},
}

\bib{Medvedev1955}{article}{
      author={Medvedev, Yuri~T.},
       title={Degrees of difficulty of the mass problems},
        date={1955},
        ISSN={0002-3264},
     journal={Doklady Akademii Nauk SSSR},
      volume={104},
       pages={501\ndash 504},
      review={\MR{0073542}},
}

\bib{Montalban:Book}{unpublished}{
      author={Montalb\'an, Antonio},
       title={{Computable Structure Theory: Part I}},
        date={2018},
        note={draft},
}

\bib{Montalban:Comm}{misc}{
      author={Montalb\'an, Antonio},
       title={Personal communication},
        date={2018},
}

\bib{Moschovakis:Descriptive}{book}{
      author={Moschovakis, Yiannis~N.},
       title={Descriptive set theory},
     edition={Second Edition},
      series={Mathematical Surveys and Monographs},
   publisher={American Mathematical Society, Providence, RI},
        date={2009},
      volume={155},
        ISBN={978-0-8218-4813-5},
         url={https://doi.org/10.1090/surv/155},
      review={\MR{2526093}},
}

\bib{Rogers:Agenda}{incollection}{
      author={Rogers, Hartley, Jr.},
       title={Some problems of definability in recursive function theory},
        date={1967},
   booktitle={Sets, {M}odels and {R}ecursion {T}heory ({P}roc. {S}ummer
  {S}chool {M}ath. {L}ogic and {T}enth {L}ogic {C}olloq., {L}eicester, 1965)},
   publisher={North-Holland, Amsterdam},
       pages={183\ndash 201},
      review={\MR{0223237}},
}

\bib{Rogers:Book}{book}{
      author={Rogers, Hartley, Jr.},
       title={Theory of recursive functions and effective computability},
   publisher={MIT Press, Cambridge, MA},
        date={1987},
        ISBN={0-262-68052-1},
      review={\MR{886890}},
}

\bib{Simpson}{article}{
      author={Simpson, Stephen~G.},
       title={Mass problems and randomness},
        date={2005},
        ISSN={1079-8986},
     journal={Bulletin of Symbolic Logic},
      volume={11},
      number={1},
       pages={1\ndash 27},
         url={https://doi.org/10.2178/bsl/1107959497},
      review={\MR{2125147}},
}

\bib{Skvortsova:intuitionismEng}{article}{
      author={Skvortsova, Elena~Z.},
       title={A faithful interpretation of the intuitionistic propositional
  calculus by means of an initial segment of the {M}edvedev lattice},
        date={1988},
        ISSN={0037-4474},
     journal={Siberian Mathematical Journal},
      volume={29},
      number={1},
       pages={133\ndash 139},
         url={https://doi.org/10.1007/BF00975025},
      review={\MR{936795}},
}

\bib{Sorbi-filters}{article}{
      author={Sorbi, Andrea},
       title={On some filters and ideals of the {M}edvedev lattice},
        date={1990},
        ISSN={0933-5846},
     journal={Archive for Mathematical Logic},
      volume={30},
      number={1},
       pages={29\ndash 48},
         url={https://doi.org/10.1007/BF01793784},
      review={\MR{1074447}},
}

\bib{Sorbi:Someremarks}{article}{
      author={Sorbi, Andrea},
       title={Some remarks on the algebraic structure of the {M}edvedev
  lattice},
        date={1990},
        ISSN={0022-4812},
     journal={Journal of Symbolic Logic},
      volume={55},
      number={2},
       pages={831\ndash 853},
         url={https://doi.org/10.2307/2274668},
      review={\MR{1056392}},
}

\bib{Sorbi:Brouwer}{article}{
      author={Sorbi, Andrea},
       title={Embedding {B}rouwer algebras in the {M}edvedev lattice},
        date={1991},
        ISSN={0029-4527},
     journal={Notre Dame Journal of Formal Logic},
      volume={32},
      number={2},
       pages={266\ndash 275},
         url={https://doi.org/10.1305/ndjfl/1093635751},
      review={\MR{1123000}},
}

\bib{SorbiSurvey}{incollection}{
      author={Sorbi, Andrea},
       title={The {M}edvedev lattice of degrees of difficulty},
        date={1996},
   booktitle={{C}omputability, {E}numerability, {U}nsolvability -- {D}irections
  in {R}ecursion {T}heory},
      editor={Cooper, S.~B.},
      editor={Slaman, T.~A.},
      editor={Wainer, S.~S.},
      series={London Mathematical Society Lecture Note Series},
      volume={224},
   publisher={Cambridge University Press, Cambridge},
       pages={289\ndash 312},
         url={https://doi.org/10.1017/CBO9780511629167.015},
      review={\MR{1395886}},
}

\end{biblist}
\end{bibdiv}

\end{document}